\documentclass{amsart}

\usepackage{graphicx}
\usepackage[labelsep=space]{caption}
\usepackage{amsfonts}
\usepackage{tikz-cd}
\usepackage{amsthm}
\usepackage{nccmath}
\usepackage{etex}
\usepackage{amssymb,mathtools}
\usepackage{mathrsfs}
\usepackage{dsfont}
\usepackage[all]{xy}
\usepackage{microtype}

\usepackage[utf8]{inputenc}
\usepackage[T1]{fontenc}
\usepackage{lmodern} 
\usepackage{latexsym}
\usepackage{MnSymbol}
\usepackage{nicefrac}
\usepackage{microtype}
\usepackage{color}
\usepackage{tikz-cd}
\usepackage{blindtext}
\usepackage{enumitem}
\newcommand{\subscript}[2]{$#1 _ #2$}
\usepackage{old-arrows}
\usepackage{empheq}
\usepackage{extpfeil}
\usepackage{hyperref}
\usepackage{theoremref}
\DeclareMathOperator{\Image}{Im}
\usepackage{float}
\makeatletter
\renewenvironment{proof}[1][\proofname]{%
	\par\pushQED{\qed}\normalfont%
	\topsep6\p@\@plus6\p@\relax
	\trivlist\item[\hskip\labelsep\bfseries#1\@addpunct{.}]%
	\ignorespaces
}{%
	\popQED\endtrivlist\@endpefalse
}
\makeatother

\usepackage{geometry}
\usepackage{secdot}
\usetikzlibrary{decorations.markings}
\tikzset{double line with arrow/.style args={#1,#2}{decorate,decoration={markings,%
			mark=at position 0 with {\coordinate (ta-base-1) at (0,1pt);
				\coordinate (ta-base-2) at (0,-1pt);},
			mark=at position 1 with {\d[#1] (ta-base-1) -- (0,1pt);
				\d[#2] (ta-base-2) -- (0,-1pt);
}}}}
\usepackage{ifxetex}
\usepackage{etoolbox}
\ifxetex
\usepackage{unicode-math}
\makeatletter
\patchcmd{\arrowfill@}{-7mu}{-14mu}{}{}
\patchcmd{\arrowfill@}{-7mu}{-14mu}{}{}
\patchcmd{\arrowfill@}{-2mu}{-4mu}{}{}
\patchcmd{\arrowfill@}{-2mu}{-4mu}{}{}
\makeatother
\fi

\theoremstyle{plain}
\newtheorem{theorem}{Theorem}[section]
\newtheorem{lemma}[theorem]{Lemma}
\newtheorem{corollary}[theorem]{Corollary}
\newtheorem{proposition}[theorem]{Proposition}
\theoremstyle{remark}
\newtheorem{remark}[theorem]{Remark}
\theoremstyle{definition}
\newtheorem{definition}[theorem]{Definition}

\newtheorem{theoremx}{Theorem}

\newtheorem{propositionx}[theoremx]{Proposition}

\begin{document}
\title{THE UNRESTRICTED VIRTUAL BRAID GROUPS $UVB_n$}

\author{Stavroula Makri}
\address{Normandie Univ., UNICAEN, CNRS, LMNO, 14000 Caen, France}
\email{stavroula.makri@unicaen.fr}

\begin{abstract}
Let $UVB_n$ and $UVP_n$ be the unrestricted virtual braid group and the unrestricted virtual pure braid group on n strands respectively. We study the groups $UVB_n$ and  $UVP_n$, and our main results are as follows: for $n\geq 5$, we give a complete description, up to conjugation, to all possible homomorphisms from $UVB_n$ to the symmetric group $S_n$. For $n\geq 3$, we characterise all possible images of $UVB_n$, under a group homomorphism, to any finite group $G$. For $n\geq 5$, we prove that $UVP_n$ is a characteristic subgroup of $UVB_n$. 
 In addition, we determine the automorphism group of $UVP_n$ and we prove that $\mathbb{Z}_2\times\mathbb{Z}_2$ is a subgroup of the outer automorphism group of $UVB_n$. Lastly, we show that $UVB_n$ and $UVP_n$ are residually finite and Hopfian but not co-Hopfian. We also remark that some of these results hold accordingly for the welded braid group $WB_n$ and we discuss about its automorphism group.
\end{abstract}
\maketitle	
	
\section{Introduction}
The group of unrestricted virtual braids, which we will denote throughout this article by $UVB_n$, was introduced by Kauffman and Lambropoulou in \cite{kauffman2004virtual} and \cite{kauffman2006virtual}, where they provide a new method for converting virtual knots and links to virtual braids and they prove a Markov Theorem for the virtual braid groups. The group $UVB_n$ also appears in \cite{kadar2017local} as a quotient of the welded braid group $WB_n$, which is a $3$-dimensional analogue of the Artin braid groups $B_n$,  and moreover in \cite{bardakov2015unrestricted} where Bardakov--Bellingeri--Damiani give a description of the structure of this group. 

A generating set of $UVB_n$ is $\langle  \sigma_1, \dots, \sigma_{n-1}, \rho_1, \dots, \rho_{n-1} \rangle$, where the generators $\sigma_1, \dots, \sigma_{n-1}$ are the standard Artin generators and the generators $\rho_1, \dots, \rho_{n-1}$ generate the symmetric group $S_n$ in $UVB_n$. Consider the map $\phi:UVB_n\to S_n$ defined by $\phi(\sigma_i)=\phi(\rho_i)=(i,i+1)\in S_n$, for $i=1,\dots,n-1$. The unrestricted virtual pure braid group, which we denote by $UVP_n$, is the kernel of the map $\phi$. In [\cite{bardakov2015unrestricted}, Theorem 2.4] it was proved that the group $UVB_n$ is isomorphic to the semi-direct product $UVP_n\rtimes S_n$, where $S_n$ acts by permuting the indices of the generators of $UVP_n$. In [\cite{bardakov2015unrestricted}, Theorem 2.7], Bardakov--Bellingeri--Damiani gave a presentation of the unrestricted virtual pure braid group, from which one can see that $UVP_n$ is isomorphic to the direct product of $n(n-1)/2$ copies of the free group of rank 2. Thus, it follows that the group $UVP_n$ is a right-angled Artin group. For a general survey on the right-angled Artin groups we direct the reader to the article \cite{charney2007introduction} by Charney.

The aim of this paper is to determine all possible homomorphisms, up to conjugation, from $UVB_n$ to the symmetric group $S_n$, and
all possible images of $UVB_n$, under a group homomorphism, to any finite group $G$. In addition, we aim to give a complete description of the automorphism group of the unrestricted virtual pure braid group $UVP_n$, and also to show that the groups $UVB_n$ and $UVP_n$ are Hopfian but not co-Hopfian. The automorphism group of $UVB_n$ has not yet been determined, but we intend to give some partial results about it.

We recall the following definitions, which we use in what follows.
\begin{itemize}
	\item Let $G, H$ be two groups. For every $x\in H$ we have the group homomorphism $h_x:H\to H$, defined by $h_x(y)=xyx^{-1}$.
	Two group homomorphisms $h_1, h_2: G\to H$ are said to be conjugate if there exists an element $x\in H$ such that $h_2=h_x\circ h_1$, which means that $h_2(g)=xh_1(g)x^{-1}$, for every $g\in G$.
	\item A group homomorphism $h:G\to H$ is said to be Abelian if its image $h(G)$ is an Abelian subgroup of $H$.
	\item A group homomorphism $h:G\to H$ is said to be cyclic if its image $h(G)$ is a cyclic subgroup of $H$.
\end{itemize}

Our main results are the following:

Let $\phi$ be the homomorphism $UVB_n\to S_n$ defined in the beginning of this section.
By $v_6$ we denote the outer automorphism of the symmetric group $S_6$. 
\begin{theoremx}\label{th1}
	Let $n\geq 5$ and let $h:UVB_n\to S_n$ be any homomorphism. Then, up to conjugation, one of the following holds:
	\begin{itemize}
		\item The homomorphism $h$ is the homomorphism $\phi$.
		\item The homomorphism $h$ is Abelian.
		\item For $n=6$, the homomorphism $h$ is $v_6\circ\phi$.
	\end{itemize}
\end{theoremx}
\noindent
We prove Theorem \ref{th1} in Section \ref{s3}. This result will be the main component in showing that, for $n\geq 5$, $UVP_n$ is a characteristic subgroup of $UVB_n$, see Proposition \ref{char}. Note that in Remark \ref{ex2}, we exhibit an example where we show that this is not the case for $n=2$. Moreover, in Proposition \ref{centralizer}, we show that the centraliser of $UVP_n$ in $UVB_n$ is trivial. 
The results that we obtain in Section \ref{s3} about $UVB_n$ hold also for $WB_n$, using similar arguments, see Remark \ref{sim1}. Moreover, those in Subsection \ref{sub1} about $UVP_n$ hold also for the welded pure braid group $WP_n$, which is the kernel of the map $\phi:WB_n\to S_n$ defined by $\phi(\sigma_i)=\phi(\rho_i)=(i,i+1)\in S_n$, for $i=1,\dots,n-1$, see Remarks \ref{sim2}, \ref{sim3}. Note that $WB_n$ and $UVB_n$ admit the same generating set, the one exhibited in the beginning of the section.
\begin{theoremx}\label{th2}
	Let $n\geq3$ and let $\phi:UVB_n\to G$ be a group homomorphism to a finite group $G$. Then, one of the following holds:
	\begin{itemize}
		\item The homomorphism $\phi$ is Abelian and in particular, $\phi(UVB_n)\cong\mathbb{Z}_m\times\mathbb{Z}_2$, for some $m\in \mathbb{Z}$ ($\mathbb{Z}_1$ is defined as the trivial group).
		\item The homomorphism $\phi$ is non-Abelian and in particular, either $|\phi(UVB_n)|\geq 2^{\frac{n(n-1)}{2}-1}\big(\frac{n(n-1)}{2}\big)!$ or $\phi(UVB_n)\cong\mathbb{Z}_m\times \phi(S_n)$, for some $m\in\mathbb{Z}$.
	\end{itemize}
\end{theoremx}
\noindent
The main tool that we use in order to determine all possible images of $UVB_n$, under a group homomorphism, in any finite group $G$ is the theory of totally symmetric sets, which was introduced by Kordek and Margalit in \cite{kordek2019homomorphisms}. In Section \ref{s4}, we introduce the notion of totally symmetric sets and we exhibit necessary results, which we will use to prove Theorem \ref{th2}. 

\begin{theoremx}\label{th3}
	Let $n\geq 2$ and $1\leq i\neq j\leq n$. It holds that $$Aut(UVP_n)\cong\langle T_{\lambda_{j,i}},\ \mathbb{Z}_2^{n(n-1)/2},\ \mathbb{Z}_2^{n(n-1)/2}\rtimes S_{n(n-1)/2} \rangle,$$
where $T_{\lambda_{j,i}}:\lambda_{i,j}\mapsto \lambda_{i,j}\lambda_{j,i},\ \text{while fixing the rest generators}$, and $S_{n(n-1)/2}$ is the symmetric group of degree $n(n-1)/2$, which acts on $\mathbb{Z}_2^{n(n-1)/2}$ by permuting the components of the product.
\end{theoremx}
\noindent
In Section \ref{s5}, we introduce the notion of right-angled Artin groups and the theory around the automorphisms of graph groups. Laurence \cite{laurence1995generating}, who extended the work of Servatius \cite{servatius1989automorphisms}, provided a complete set of generators for the automorphism group of any graph group. We will use this result to prove Theorem \ref{th3}.

Finally, in Section \ref{s6}, we give some partial results about the automorphism group of $UVB_n$, see Corollary \ref{colrose}, as well as the following proposition.

Let $\beta_n, \gamma_n:UVB_n\to UVB_n$ such that, for $1\leq i\leq n-1$, $\beta_n:\sigma_i\mapsto\sigma_i^{-1},\ \rho_i\mapsto\rho_i$ and $\gamma_n:\sigma_i\mapsto\rho_i\sigma_i\rho_i,\ \rho_i\mapsto\rho_i$.
\begin{propositionx}\label{prop1}
	Let $n\geq 3$. It holds that $$\langle\beta_n, \gamma_n\rangle \subseteq  Out(UVB_n),\ \text{where}\ \langle\beta_n, \gamma_n\rangle\cong \mathbb{Z}_2\times\mathbb{Z}_2.$$
\end{propositionx}
\noindent
Note that in \cite{paris}, Bellingeri--Paris proved, for the virtual braid groups $VB_n$, that $Out(VB_n)\cong \mathbb{Z}_2\times\mathbb{Z}_2$. We will not define here the virtual braid groups, but we could say that they are an extension of the classical braid groups by the symmetric
group. Since it holds for the welded group $WB_n$ that it is a quotient of the virtual group $VB_n$ and $UVB_n$ a quotient of $WB_n$, we thus speculate that the result by Bellingeri--Paris about $Out(VB_n)$ together with our partial result about $Out(UVB_n)$ could be of help for future work in determining the group $Out(WB_n)$, which is still an open problem. In Section \ref{s6}, we discuss about $Aut(WB_n)$, and we conjecture that, for $n\geq 3$, $Out(WB_n)\cong\mathbb{Z}_2$. We conclude this section with Corollary \ref{hop}, where we show that $UVB_n$ and $UVP_n$ are residually finite and Hopfian but not co-Hopfian.

\section{Homomorphisms from $UVB_n$ to the symmetric group $S_n$ and the group $UVP_n$}\label{s3}
In order to define the unrestricted virtual braid groups, we will first
introduce welded braid groups by simply recalling their group presentation. Nevertheless, for other more substantial definitions of welded braid groups, we refer the reader, for instance, to \cite{brendle2013configuration} and \cite{fenn1997braid}.

\begin{definition}\label{def1}
	Let $n\in \mathbb{N}$. The welded braid group $WB_n$ is defined by the group presentation
	$$\langle \sigma_1, \dots, \sigma_{n-1}, \rho_1, \dots, \rho_{n-1}\ |\ R  \rangle,$$
	where $R$ is the set of the following relations:
	\begin{enumerate}[label=(\subscript{R}{{\arabic*}})]
		\item $\sigma_i\sigma_{i+1}\sigma_i=\sigma_{i+1}\sigma_i\sigma_{i+1}$, for $i=1, \dots, n-2$,
		\item $\sigma_i\sigma_j=\sigma_j\sigma_i$, for $|i-j|>1$,  where $1\leq i,j\leq n-1$,
		\item $\rho_i\rho_{i+1}\rho_i=\rho_{i+1}\rho_i\rho_{i+1}$, for $i=1, \dots, n-2$,
		\item $\rho_i\rho_j=\rho_j\rho_i$, for $|i-j|>1$, where $1\leq i,j\leq n-1$, 
		\item $\rho_i^2=1$, for $i=1, \dots, n-1$,
		\item $\sigma_i\rho_j=\rho_j\sigma_i$, for $|i-j|>1$,  where $1\leq i,j\leq n-1$,
		\item $\sigma_i\rho_{i+1}\rho_i=\rho_{i+1}\rho_i\sigma_{i+1}$, for $i=1, \dots, n-2$,\label{rel7}
		\item $\rho_i\sigma_{i+1}\sigma_i=\sigma_{i+1}\sigma_i\rho_{i+1}$, for $i=1, \dots, n-2$.\label{rel8}
	\end{enumerate}
\end{definition}
\begin{remark}
Note that, for $i=1, \dots, n-2$, the symmetrical relations of \ref{rel7}, $\rho_i\rho_{i+1}\sigma_i=\sigma_{i+1}\rho_i\rho_{i+1}$, also hold in $WB_n$.
\end{remark}

In [\cite{bardakov2015unrestricted}, Remark 2.2], we see that, for $i=1, \dots, n-2$, the symmetrical relations of \ref{rel8},
$\rho_{i+1}\sigma_i\sigma_{i+1}=\sigma_i\sigma_{i+1}\rho_i$, do not hold in $WB_n$. The unrestricted virtual braid group $UVB_n$ is defined so that these symmetrical relations hold.

\begin{definition}\label{defuvbn}
	Let $n\in \mathbb{N}$. The unrestricted virtual braid group $UVB_n$ is defined as the quotient of $WB_n$, which is defined in Definition \ref{def1}, by the following relation:
	$$(R_9)\ \rho_{i+1}\sigma_i\sigma_{i+1}=\sigma_i\sigma_{i+1}\rho_i,\ \text{for}\ i=1, \dots, n-2.$$
		
\end{definition}

\begin{remark}\label{useouter}
	Let $n\geq 2$. Based on the presentation of $UVB_n$, Definition \ref{defuvbn}, it follows that the Abelianisation of $UVB_n$ is isomorphic to $\mathbb{Z}\times\mathbb{Z}_2$, where $\mathbb{Z}$ is generated by the image of $\sigma_1$ and $\mathbb{Z}_2$ is generated by the image of $\rho_1$.
\end{remark}

It is well known that the symmetric group $S_6$ has an outer automorphism, which we will denote by $v_6$, unlike all other symmetric groups.
Due to Artin, \cite{artin1947braids} and Lin, \cite{lin2004braids}, the following known result can be deduced.
\begin{proposition}\label{sn}
	Let $n,m\in \mathbb{Z}$ with $n\geq m$, such that $n\geq5$, $m\geq2$. For any homomorphism $h:S_n\to S_m$ one of the following holds.
	\begin{enumerate}
		\item The homomorphism $h$ is Abelian and therefore cyclic.
		\item For $n=m$ the homomorphism $h$ is, up to conjugation, the identity.
		\item For $n=m=6$ the homomorphism $h$ is, up to conjugation, $v_6$.
	\end{enumerate}
\end{proposition}

We define the following maps that we will use in the proof of Theorem \ref{th1}.

\begin{itemize}
	\item Let $\alpha$ be the homomorphism $S_n\to UVB_n$ defined by $\alpha(s_i)=\rho_i$, for $1\leq i\leq n-1$.
	\item Let $\phi$ be the homomorphism $UVB_n\to S_n$ defined by $\phi(\sigma_i)=\phi(\rho_i)=s_i$, where $s_i=(i, i+1)\in S_n$, for every $1\leq i\leq n-1$.
\end{itemize}

\begin{proof}[Proof of Theorem \ref{th1}]
	For $n\geq 5$, let $h: UVB_n\to S_n$ be any homomorphism. 
	From Proposition \ref{sn}, the composition map $h\circ\alpha:S_n\to S_n$ is, up to conjugation, either Abelian or the identity homomorphism, and only in the case $n=6$ we could also have that $h\circ\alpha$ is the homomorphism $v_6$. We will examine these cases separately. 
	
	Suppose that $h\circ\alpha$ is the identity homomorphism. We have that
	$(h\circ\alpha)(s_i)=h(\alpha(s_i))=h(\rho_i)=s_i$, for $1\leq i\leq n-1$. It follows that $h(\rho_i)=s_i$, for $1\leq i\leq n-1$. Moreover, from relation $\rho_i\sigma_j=\sigma_j\rho_i,\ |i-j|>1$, we get that $h(\sigma_1)=h(\rho_i^{-1}\sigma_1\rho_i)=s_i^{-1}h(\sigma_1)s_i$, for $3\leq i\leq n-1$, which means that $s_ih(\sigma_1)=h(\sigma_1)s_i$, for $3\leq i\leq n-1$. Thus, $h(\sigma_1)$ belongs to the centraliser of $\langle s_3,\dots,s_{n-1}\rangle$ inside $S_n$, but this centraliser is $\{1, s_1\}$. As a result, either $h(\sigma_1)=1$ or $h(\sigma_1)=s_1$. We shall check each case separately.
	
	Suppose that $h(\sigma_1)=1$. In this case, from $\sigma_i=(\sigma_1\dots\sigma_{n-1})^{i-1}\sigma_1(\sigma_1\dots\sigma_{n-1})^{1-i}$, we have that $h(\sigma_i)=1, \ 1\leq i\leq n-1$. Therefore, we get $h(\rho_i)=s_i$ and $h(\sigma_i)=1$, for $1\leq i\leq n-1$. Now, from relation $\rho_i\sigma_{i+1}\sigma_i=\sigma_{i+1}\sigma_i\rho_{i+1}$, for $1\leq i\leq n-2$, it follows that $h(\rho_i)h(\sigma_{i+1})h(\sigma_i)=h(\sigma_{i+1})h(\sigma_i)h(\rho_{i+1})$. Thus, $h(\rho_i)=h(\rho_{i+1})=s_i=s_{i+1}$, for $1\leq i\leq n-2$, which leads to a contradiction since we have assumed that $h(\rho_i)=s_i$, for $1\leq i\leq n-1$.
	
	Suppose that $h(\sigma_1)=s_1$. By induction we can show that $h(\sigma_i)=s_i$, for $1\leq i\leq n-1$. For $i=1$ the hypothesis holds. Suppose that for some $i\geq 2$ we have $h(\sigma_i)=s_i$.
	From relation $\rho_{i+1}\rho_i\sigma_{i+1}=\sigma_i\rho_{i+1}\rho_i$ it follows that $h(\sigma_{i+1})=h(\rho_i)^{-1}h(\rho_{i+1})^{-1}h(\sigma_i)h(\rho_{i+1})h(\rho_i)$. Thus, $h(\sigma_{i+1})=s_i^{-1}s_{i+1}^{-1}s_is_{i+1}s_i=s_{i+1}$, which completes the induction. As a result we have that $h(\rho_i)=s_i$ and $h(\sigma_i)=s_i$, for $1\leq i\leq n-1$. This implies that $h$ is, up to conjugation, the homomorphism $\phi$.
	
	We will consider now the case where the homomorphism $h\circ\alpha:S_n\to S_n$ is Abelian, and since, for $1\leq i\leq n-1$, it holds that $s_is_{i+1}s_i=s_{i+1}s_is_{i+1}$ in $S_n$, it follows that $h\circ\alpha$ is also cyclic. Thus, for $1\leq i\leq n-1$, $(h\circ\alpha)(s_i)=w$, for an element $w\in S_n$, with $w^2=1$. Therefore, $(h\circ\alpha)(s_i)=h(\alpha(s_i))=h(\rho_i)=w$, which means that, for $1\leq i\leq n-1$, $h(\rho_i)=w$, where $w^2=1$. Relation $\sigma_i\rho_{i+1}\rho_i=\rho_{i+1}\rho_i\sigma_{i+1}$, under $h$, becomes $h(\sigma_i)w^2=w^2h(\sigma_{i+1})$, and therefore, for $1\leq i\leq n-1$, $h(\sigma_i)=h(\sigma_{i+1})$. Thus, we have $h(\sigma_1)=\dots=h(\sigma_{n-1})=:\tau \in S_n$.
	Moreover, from relation $\rho_i\sigma_j=\sigma_j\rho_i$, for $|i-j|>1$ and $1\leq i, j\leq n-1$, we get $h(\rho_i)h(\sigma_j)=h(\sigma_j)h(\rho_i)$, which implies that $w\tau=\tau w$.
	As a result we have that the image of the homomorphism $h$ is the Abelian group generated by the elements $w$ and $\tau$, with $w^2=1$.
	
	Lastly, suppose that $n=6$ and that the homomorphism $h\circ\alpha:S_6\to S_6$ is, up to conjugation, the homomorphism $v_6$; $h\circ\alpha=v_6$. The map $v_6^{-1}\circ h\circ\alpha$ becomes the identity homomorphism $S_6\to S_6$ and it follows that $(v_6^{-1}\circ h\circ\alpha)(s_i)=(v_6^{-1}\circ h)(\rho_i)=s_i$, for $1\leq i\leq 5$. Using relation $\rho_i\sigma_j=\sigma_j\rho_i,\ |i-j|>1$ we get that, for $i\in\{3,4,5\}$, $(v_6^{-1}\circ h)(\sigma_1)=(v_6^{-1}\circ h)(\rho_i^{-1}\sigma_1\rho_i)=
	s_i^{-1}(v_6^{-1}\circ h)(\sigma_1)s_i$, and we obtain $s_i(v_6^{-1}\circ h)(\sigma_1)=(v_6^{-1}\circ h)(\sigma_1)s_i$. In other words, we have that $(v_6^{-1}\circ h)(\sigma_1)$ belongs to the centraliser of $\langle s_3,s_4,s_5\rangle$ in $S_6$, but the centraliser of $\langle s_3,s_4,s_5\rangle$ in $S_6$ is $\{1,s_1\}$. Therefore, either $(v_6^{-1}\circ h)(\sigma_1)=1$ or $(v_6^{-1}\circ h)(\sigma_1)=s_1$.  
	Following the same arguments as before, in the case where $(v_6^{-1}\circ h)(\sigma_1)=1$, we conclude that, for $1\leq i\leq n-2$,  $(v_6^{-1}\circ h)(\sigma_i)=1$ and that $(v_6^{-1}\circ h)(\rho_i)=(v_6^{-1}\circ h)(\rho_{i+1})=s_i=s_{i+1}$. Thus, $(v_6^{-1}\circ h)$ is a cyclic homomorphism, and as a result $h$ is again a cyclic homomorphism, whose image is of order $2$. In the case where $(v_6^{-1}\circ h)(\sigma_1)=s_1$, it follows that $(v_6^{-1}\circ h)(\sigma_i)=s_i$, and as a result $(v_6^{-1}\circ h)=\phi$, since we already have that $(v_6^{-1}\circ h)(\rho_i)=s_i$, for $1\leq i\leq 5$. All together we get that the homomorphism $h$ can be $v_6\circ \phi$.
\end{proof}

\begin{remark}\label{r1}
	For $n=3, 4$, to determine all possible homomorphisms from $UVB_n$ to $S_n$ seems to be trickier. Note that for $n=2$ it holds that $UVB_2=\langle\sigma_1, \rho_1 \ |\ \rho_1^2=1\rangle= WB_2$, which is isomorphic to $\mathbb{Z}\ast\mathbb{Z}_2$, and thus the image of any homomorphism $h: UVB_2\to S_2$ is either the trivial group or $S_2$.
\end{remark}

\begin{remark}\label{surj}
	The only possible surjective homomorphisms from $UVB_n$ to $S_n$ , for $n\geq 5$, are the homomorphism $\phi$ and also the homomorphism $v_6\circ\phi$ when $n=6$.
\end{remark}

\begin{remark}\label{sim1}
Note that, using similar arguments, Theorem \ref{th1}, Remark \ref{r1} and \ref{surj} hold also for $WB_n$.
\end{remark}
\subsection{The unrestricted virtual pure braid group $UVP_n$}\label{sub1}
\hfill\\

We shall study now, a subgroup of $UVB_n$, the unrestricted virtual pure braid group.
We recall that $\phi:UVB_n\to S_n$ is defined by $\phi(\sigma_i)=\phi(\rho_i)=(i,i+1)\in S_n$, for $i=1,\dots,n-1$. The unrestricted virtual pure braid group, which we denote by $UVP_n$, is the kernel of the map $\phi$. 

In \cite{bardakov2015unrestricted}, Bardakov--Bellingeri--Damiani defined the following elements of $UVP_n$:
\begin{equation}\label{lamda}
\begin{aligned}
&\lambda_{i,i+1}=\rho_i\sigma_i^{-1},\ \text{for}\ i=1,\dots,n-1,\\
&\lambda_{i+1,i}=\sigma_i^{-1}\rho_i,\ \text{for}\ i=1,\dots,n-1,\\
&\lambda_{i,j}=\rho_{j-1}\rho_{j-2}\dots\rho_{i+1}\lambda_{i,i+1}\rho_{i+1}\dots\rho_{j-2}\rho_{j-1},\ \text{for}\ 1\leq i<j-1\leq n-1,\\
&\lambda_{j,i}=\rho_{j-1}\rho_{j-2}\dots\rho_{i+1}\lambda_{i+1,i}\rho_{i+1}\dots\rho_{j-2}\rho_{j-1},\ \text{for}\ 1\leq i<j-1\leq n-1.
\end{aligned}
\end{equation}
Moreover, they gave a presentation of $UVP_n$ as follows.

\begin{theorem}[Bardakov--Bellingeri--Damiani, \cite{bardakov2015unrestricted}]\label{tss}
	Let $n\in \mathbb{N}$. The group $UVP_n$ admits the following presentation:\\
	Generators: $\lambda_{i,j}$, for $1\leq i\neq j\leq n$.\\
	Relations: The generators pairwise commute except for the couples $\lambda_{i,j}, \lambda_{j,i}.$ 
\end{theorem}

\begin{remark}\label{free}
	The group $$UVP_n=\langle\lambda_{i,j}, \ 1\leq i\ne j \leq n \ | \ \lambda_{i,j}\lambda_{k,l}=\lambda_{k,l}\lambda_{i,j}, \ \text{for}\ (k,l)\neq(j,i),\ 1\leq i,j,k,l\leq n\rangle,$$
	can be seen as the direct product of the following $n(n-1)/2$ factors:
	$$UVP_n=\langle\lambda_{1,2}, \lambda_{2,1}\rangle\times\dots\times\langle\lambda_{i,j}\lambda_{j,i}\rangle\times\dots\times\langle \lambda_{n-1,n},\lambda_{n,n-1}\rangle, \ \text{for}\ 1\leq i\neq j\leq n.$$
	Thus, $UVP_n$ is isomorphic to the direct product of $n(n-1)/2$ copies of the free group of rank 2:
	$$UVP_n\cong\underbrace{F_2\times\dots\times F_2\times\dots\times F_2}_\text{$n(n-1)/2$-times}, \ \text{for}\ n\geq 2.$$\\
\end{remark}
\begin{remark}\label{center}
	The group $UVP_n$ has trivial centre, $Z(UVP_n)=e$, since it is isomorphic to the direct product of free groups. For the same reason it follows that $UVP_n$ is torsion free as well.
\end{remark}

Based on the presentation of $UVP_n$, given in Theorem \ref{tss}, where the relations are commutation relations, we obtain the following result about the Abelianisation of $UVP_n$. 

\begin{corollary}\label{abuvpn}
	For $n\geq2$, the Abelianisation of $UVP_n$ is isomorphic to $\mathbb{Z}^{n(n-1)}$.
\end{corollary}

The question that was posed in \cite{kauffman2006virtual} about the non-trivial structure of $UVB_n$ was answered by Bardakov--Bellingeri--Damiani in \cite{bardakov2015unrestricted}, where they gave a decomposition of $UVB_n$ into its subgroup $UVP_n$ and the symmetric group $S_n$, as presented in the following theorem.

\begin{theorem}[Bardakov--Bellingeri--Damiani, \cite{bardakov2015unrestricted}]\label{tss2}
	The group $UVB_n$ is isomorphic to the semi-direct product $UVP_n\rtimes S_n$, where $S_n$ acts by permuting the indices of the generators of $UVP_n$.
\end{theorem} 

More precisely, for all $\lambda_{i,j}\in UVP_n$, where $1\leq i\neq j\leq n$, and for any $s\in S_n$ we have the following conjugating rule: 
\begin{equation}\label{conj}
\iota(s)\lambda_{i,j}\iota(s)^{-1}=\lambda_{s(i),s(j)},
\end{equation}
 where $\iota$ is the injective map $\iota: S_n\to UVB_n$ defined by $\iota\big((i,i+1)\big)=\rho_i$; it is the natural section for the map $\phi$. Moreover, the action of the symmetric group $S_n$ on the generating set of $UVP_n$ is transitive, see [\cite{bardakov2015unrestricted}, Corollary 2.6].

\begin{remark}
	Having that $UVB_n\cong UVP_n\rtimes S_n$ it follows that the centre of $UVB_n$ is trivial, $Z(UVB_n)=e$, since the centre of the symmetric group and the centre of $UVP_n$ are trivial, see Remark \ref{center}.
\end{remark}

Using Theorem \ref{th1}, we will prove that, for $n\geq 5$, the group $UVP_n$ is actually a characteristic subgroup of $UVB_n$.

\begin{proposition}\label{char}
	For $n\geq 5$, the group $UVP_n$ is a characteristic subgroup of $UVB_n$.
\end{proposition} 

\begin{proof}
	In order to prove that the group $UVP_n$ is a characteristic subgroup of $UVB_n$, we have to show that for any automorphism of $UVB_n$ the subgroup $UVP_n$ stays invariant. That is for any $f\in Aut(UVB_n)$ then $f(UVP_n)=UVP_n$.
	
	Let $f$ be any automorphism of the group $UVB_n$ and $h$ be any surjective homomorphism from $UVB_n$ to the symmetric group $S_n$. We consider the following composition map:
    \begin{equation*}
	\begin{tikzcd}
      h\circ f: UVB_n\ar[r, "f"]&UVB_n\ar[r, two heads, "h"]&S_n.
	\end{tikzcd}
	\end{equation*}
		
	From Remark \ref{surj}, we know that the only possible surjective homomorphisms from $UVB_n$ to $S_n$ are the homomorphism $\phi$ and, in the case where $n=6$, also the homomorphism $v_6\circ\phi$. Note that these homomorphisms have kernel the group $UVP_n$. Therefore, it follows that $\ker(h)=UVP_n$ and $\ker(h\circ f)=UVP_n$, as both of them are epimorphisms from $UVB_n$ to $S_n$.
	
	We have that $\ker(h\circ f)=f^{-1}(UVP_n)$, since $f$ is an automorphism and $\ker(h)=UVP_n$. Moreover, we have that $\ker(h\circ f)=UVP_n$. It follows that $f^{-1}(UVP_n)=UVP_n$, for any $f\in Aut(UVB_n)$, which completes the proof.
\end{proof}
\begin{remark}\label{ex2}
	For $n=2$ the group $UVP_2$ is not a characteristic subgroup of $UVB_2$. This is the case because the automorphism $\alpha:UVB_2\to UVB_2$, defined by \begin{equation*}
	\alpha: 
	\begin{cases}
	\sigma_1\mapsto \sigma_1^{-1}\rho_1,\\
	\rho_1\mapsto \rho_1,
	\end{cases}
	\end{equation*}
	does not send the element $\lambda_{1,2}\in UVP_2$ to an element in $UVP_2$. In particular, $\lambda_{1,2}=\rho_1\sigma_1^{-1}$, from Remark \ref{lamda}, and thus $\alpha(\lambda_{1,2})=\sigma_1\nin UVP_2$.
\end{remark}
\begin{remark}\label{sim2}
	Note that, using similar arguments, Proposition \ref{char} and Remark \ref{ex2} hold also for the welded pure braid group $WP_n$, which is a subgroup of $WB_n$, and in particular the kernel of the map $\phi:WB_n\to S_n$, defined by $\phi(\sigma_i)=\phi(\rho_i)=(i,i+1)\in S_n$, for $i=1,\dots,n-1$. Moreover, the generating set of $WP_n$ is the same as the one of $UVP_n$ and thus Relations \eqref{lamda}, considering the elements $\lambda_{i,j}$ as generators of $WP_n$, hold for $WP_n$ as well. For more details about the group $WP_n$, we refer the reader, for instance to \cite{savushkina1996group} and \cite{brendle2013configuration}.
\end{remark}
Before concluding this section, we show that $UVP_n$ has trivial centraliser in $UVB_n$. 
\begin{proposition}\label{centralizer}
	Let $n\in \mathbb{N}$. The centraliser of $UVP_n$ in $UVB_n$ is trivial.
\end{proposition}
\begin{proof}
	Let $n\in \mathbb{N}$. We want to calculate the centraliser $$C_{UVB_n}(UVP_n)=\{g\in UVB_n\ | \ gp=pg,\ \text{for every}\ p\in UVP_n \}.$$ We know that $UVB_n=UVP_n\rtimes S_n$, which means that any $g\in UVB_n$ can be expressed as $g=u(\lambda_{i,j})w(s_k)$, for $u(\lambda_{i,j})$ and $w(s_k)$ a word in $UVP_n$ and $S_n$ respectively.
	
	Let $g\in C_{UVB_n}(UVP_n)$ be a non-trivial element, where $g=E S$, for $E, S$ fixed words in $UVP_n\ \text{and}\ S_n$ respectively. It holds that $gp=pg$ for every element $p\in UVP_n$. Therefore, 
	\begin{equation*}
	\lambda_{k,l}\cdot E S=E S\cdot \lambda_{k,l},\ \text{for every}\ k\neq l\in \{1,\dots,n\}.
	\end{equation*}
	
	Based on the action of the symmetric group on the generator $\lambda_{k,l}\in UVP_n$, Relation \eqref{conj}, we obtain
	
	$$\lambda_{k,l}\cdot E S=E\cdot\lambda_{S(k),S(l)}S,\ \text{for every}\ k\neq l\in \{1,\dots,n\}.$$
	Thus,
	\begin{equation}\label{equal2}
	\lambda_{k,l}=E\cdot\lambda_{S(k),S(l)}\cdot E^{-1} \in UVP_n,\ \text{for every}\ k\neq l\in \{1,\dots,n\}.
	\end{equation}
	
	From relation \eqref{equal2}, if $\lambda_{k,l}=\lambda_{S(k),S(l)}$, for every $k\neq l\in \{1,\dots,n\}$, it follows that the element $E$ commutes with every generator $\lambda_{k,l}\in UVP_n$, and then we obtain a contradiction, since $UVP_n$ has a trivial centre, Remark \ref{center}. Therefore, it holds that there exists $k,l\in\{1,\dots,n\}$ such that $\lambda_{k,l}\neq\lambda_{S(k),S(l)}$ and $\lambda_{k,l}=E\cdot\lambda_{S(k),S(l)}\cdot E^{-1}$. It follows that for such $k,l$ there exists a pair $(r,t)\in \{1,\dots,n\}$, such that $\big(S(k),S(l)\big)=(r,t)$ and $\lambda_{k,l}=E\cdot \lambda_{r,t}\cdot E^{-1}$. This means that under the Abelianisation map the two distinct generators of the group $UVP_n$, $\lambda_{k,l}$ and $\lambda_{r,t}$, would coincide. But, as stated in Corollary \ref{abuvpn}, the Abelianisation of $UVP_n$ is isomorphic to the free Abelian group of rank $n(n-1)$ generated by the elements $\lambda_{i, j}$, for $1\leq i\neq j\leq n$. Therefore, relation \eqref{equal2} can not hold. 
	
	Now, suppose that $g=S$; meaning that $E$ is a trivial word in $UVP_n$. It has to hold that
	\begin{equation}\label{equal3}
	\lambda_{k,l}\cdot S=S\cdot \lambda_{k,l},\ \text{for every}\ k,l\in \{1,\dots,n\}.
	\end{equation} 
	But relation \eqref{equal3} implies that the word $S$ in the symmetric group fixes all the elements of the set $\{1,\dots,n\}$. This is possible only when $S$ is the trivial element, which leads once more to a contradiction, since $g$ is a non-trivial element in  $C_{UVB_n}(UVP_n)$.
	Since the centre of $UVP_n$ is trivial we do not need to check the case where $g=E$.
	
	We conclude that there does not exist a non-trivial element $g$ in $C_{UVB_n}(UVP_n)$ and therefore the centraliser $C_{UVB_n}(UVP_n)$ is trivial.
\end{proof}

\begin{remark}\label{sim3}
	By Savushkina \cite{savushkina1996group}, we have $WB_n=WP_n\rtimes S_n$, where $S_n$ acts on $WP_n$ by permuting the indices of the generators of $WP_n$, and moreover that $WP_n$ has a trivial centre. In addition, from the presentation of $WP_n$, \cite{savushkina1996group}, which has the same generating set as $UVP_n$, and only commutation relations, it follows that the Abelianisation of $WP_n$ is also isomorphic to $\mathbb{Z}^{n(n-1)}$. Thus, using similar arguments, it follows that the centraliser of $WP_n$ in $WB_n$ is trivial.
\end{remark}

\section{Finite image of $UVB_n$}\label{s4}
The main tool that we will use in order to determine all possible images of $UVB_n$, under a group homomorphism, in any finite group $G$ is the theory of totally symmetric sets, which was introduced by Kordek and Margalit in \cite{kordek2019homomorphisms}. 

\begin{definition}[Kordek--Margalit, \cite{kordek2019homomorphisms}]
	Let $G$ be any group. A subset $X$ of $G$ is called a totally symmetric set of $G$ if it satisfies the following two conditions:
	\begin{itemize}
		\item The elements of the set $X=\{x_1,\dots,x_n\}$ pairwise commute.
		\item Each permutation of $X$ can be achieved via conjugation by an element of $G$. That is, for any permutation $s\in S_n$, there exists $g\in G$ such that $gx_ig^{-1}=x_{s(i)}$, for all $1\leq i\leq n$.
	\end{itemize}
\end{definition}

From \cite{kordek2019homomorphisms} and \cite{chudnovsky2020finite}, the following facts hold.

\begin{lemma}[Kordek--Margalit, \cite{kordek2019homomorphisms}]\label{KM}
	Let $X$ be a totally symmetric set of a group $G$. For any homomorphism $h: G\to H$ it holds that $|h(X)|$ is equal to either 1 or $|X|$.
\end{lemma}

\begin{proposition}[Chudnovsky, Kordek, Li, Partin, \cite{chudnovsky2020finite}]\label{pr1.4}
	Suppose that $X$ is a totally symmetric set of a group $G$ with $|X|=k$. If the elements of $X$ have finite order, then $|G|\geq 2^{k-1}k!$.
\end{proposition}
 
\begin{remark}
Note that all elements of a totally symmetric set $X$ are conjugate to each other, and thus
every element of $X$ has the same order. In particular, if one element of $X$ has finite order $p$, then every other element of $X$ has also order $p$. Thus,
Proposition \ref{pr1.4} can be restated in the following way: Let $G$ be any group and $H$ a finite one. Suppose that $X$ is a totally symmetric set of $G$. For any non-trivial homomorphism $h:G\to H$ it holds that $|h(G)|\geq 2^{|X|-1}|X|!$.
	From Lemma \ref{KM} we know that either $|h(X)|=|X|$ or $|h(X)|=1$. 
	Thus, the equivalence between these two statements comes from the fact that for any finitely generated Abelian group $S$, then $S$ is a torsion group if and only if the group $S$ is a finite group. 
	
\end{remark}

We are ready now to define some totally symmetric sets, for $n\geq3$, of $UVB_n$. Based on the presentation of $UVP_n$ given in Theorem \ref{tss}, we define the following totally symmetric sets.

\begin{lemma}\label{mtss}
	Let $n\geq 3$. The following $n$ sets, whose cardinality is $n(n-1)/2$, are totally symmetric sets of $UVB_n$:
	$$A_i:= \Big\{\lambda_{i,1},\dots, \lambda_{i,n}, B_i, C_i\Big\},\ \text{for} \ 1\leq i\leq n\ \text{and}\ \lambda_{i,i}=1,$$
	where $B_i=\bigcup\limits_{j=n}^{i+1}\{\lambda_{j,k}\}_{1\leq k\neq i\leq j-1}$ and $C_i=\bigcup\limits_{s=i-1}^{2}\{\lambda_{s,t}\}_{1\leq t\leq s-1}.$ 
\end{lemma}

\begin{proof}
	Let $n\geq3$. The sets $A_i$ do not contain any pair of elements of the form $\{\lambda_{k,l}, \lambda_{l,k}\}$, and therefore the elements inside each $A_i$ pairwise commute, by Theorem \ref{tss}. Moreover, by Theorem \ref{tss2}, $UVB_n$ can be seen as a semi-direct product, $UVP_n\rtimes S_n$, where the symmetric group $S_n$ acts by conjugation on the elements of $UVP_n$ permuting the set of the elements $\lambda_{i,j}$. Thus, it follows that indeed the elements $\lambda_{i,j} $ are pairwise conjugate in $UVB_n$. As a result we get that indeed the sets $A_i$ are totally symmetric sets. Moreover, the size of every totally symmetric set $A_i$ is: $$|A_i|= (n-1)+(n-2)+\dots+\big( n-(n-2)\big)+\big( n-(n-1)\big)=n(n-1)/2.$$
\end{proof}

With the following example we make clearer the construction of these totally symmetric sets in $UVB_n$. For $n=5$ we have the following five totally symmetric sets in $UVB_5$:
$$A_1=\{\lambda_{1,2},\lambda_{1,3}, \lambda_{1,4}, \lambda_{1,5}, \lambda_{5,2},\lambda_{5,3}, \lambda_{5,4}, \lambda_{4,2}, \lambda_{4,3}, \lambda_{3,2}\},\ |A_1|=10.$$
$$A_2=\{\lambda_{2,1},\lambda_{2,3}, \lambda_{2,4}, \lambda_{2,5}, \lambda_{5,1},\lambda_{5,3}, \lambda_{5,4}, \lambda_{4,1}, \lambda_{4,3}, \lambda_{3,1}\},\ |A_2|=10.$$
$$A_3=\{\lambda_{3,1},\lambda_{3,2}, \lambda_{3,4}, \lambda_{3,5}, \lambda_{5,1},\lambda_{5,2}, \lambda_{5,4}, \lambda_{4,1}, \lambda_{4,2}, \lambda_{2,1}\},\ |A_3|=10.$$
$$A_4=\{\lambda_{4,1},\lambda_{4,2}, \lambda_{4,3}, \lambda_{4,5}, \lambda_{5,1},\lambda_{5,2}, \lambda_{5,3}, \lambda_{3,1}, \lambda_{3,2}, \lambda_{2,1}\},\ |A_4|=10.$$
$$A_5=\{\lambda_{5,1},\lambda_{5,2}, \lambda_{5,3}, \lambda_{5,4}, \lambda_{4,1},\lambda_{4,2}, \lambda_{4,3}, \lambda_{3,1}, \lambda_{3,2}, \lambda_{2,1}\},\ |A_5|=10.$$
\\	
\begin{remark}\label{r1.7}
	We observe that for any $1\leq i, j\leq n$ where $i\neq j$ we have that $A_i\neq A_j$, since $\lambda_{i,j}\in A_{i}$ but $\lambda_{i,j}\nin A_{j} $, and also that $A_i\cap A_j\neq \emptyset$. Moreover, the set $\bigcup_{i=1}^nA_i$ is equal to the generating set of $UVP_n$. That is $\bigcup_{i=1}^nA_i=\{\lambda_{i,j}\}_{1\leq i\neq j\leq n}$.	
\end{remark}

We shall now prove Theorem \ref{th2}.

\begin{proof}[Proof of Theorem \ref{th2}]
	Suppose that $\phi$ is Abelian. That means that the image of $UVB_n$ is an Abelian subgroup of $G$. In this case, for $1\leq i\leq n-2$, the relation  $\sigma_i\sigma_{i+1}\sigma_i=\sigma_{i+1}\sigma_i\sigma_{i+1}$ that holds in $UVB_n$ implies $\phi(\sigma_i)\phi(\sigma_{i+1})\phi(\sigma_i)=\phi(\sigma_{i+1})\phi(\sigma_i)\phi(\sigma_{i+1})$. Since $\phi(UVB_n)$ is an Abelian subgroup of $G$ it follows that $\phi(\sigma_i)^2\phi(\sigma_{i+1})=\phi(\sigma_i)\phi(\sigma_{i+1})^2$, and therefore, for $1\leq i\leq n-2$, $\phi(\sigma_i)=\phi(\sigma_{i+1})$. Thus, there exists an element $g$ in $G$ such that, for $1\leq i\leq n-1$, $\phi(\sigma_i)=g\in UVB_n$. Similarly, for $1\leq i\leq n-2$, using the relation $\rho_i\rho_{i+1}\rho_i=\rho_{i+1}\rho_i\rho_{i+1}$ in $UVB_n$ and applying the same argument it follows that $\phi(\rho_i)=\phi(\rho_{i+1})$. Moreover, for $1\leq i\leq n-1$, it holds that $\phi(\rho_i)^2=1$. We conclude that $\phi(\sigma_i)$ generates a subgroup isomorphic to $\mathbb{Z}_m$, for some $m\in \mathbb{Z}$, and $\phi(\rho_i)$ generates a subgroup isomorphic to $\mathbb{Z}_2$ in $G$. All together, we obtain that  $$\phi(UVB_n)\cong\mathbb{Z}_m\times\mathbb{Z}_2.$$
	
	Suppose that $\phi$ is not Abelian. Moreover, suppose that at least one of the images of the totally symmetric sets $A_i$, defined in Lemma \ref{mtss}, under the homomorphism $\phi$, is not a singleton. From Lemma \ref{KM} we have that either $|\phi(A_i)|=n(n-1)/2$ or $|\phi(A_i)|=1$, and therefore, for some $k\in \{1,\dots,n\}$, we have a totally symmetric set $A_k$ for which $|\phi(A_k)|=n(n-1)/2$. From Proposition \ref{pr1.4} we obtain that $$|\phi(UVB_n)|\geq 2^{\frac{n(n-1)}{2}-1}\big(\frac{n(n-1)}{2}\big)!.$$
	
	Finally, we consider that $\phi$ is not Abelian and also that $|\phi(A_i)|=1$, for all $1\leq i\leq n$. The fact that $|\phi(A_i)|=1$ implies that, for every $1\leq i\leq n$, $\phi(A_i)=g_i\in G$. From Remark \ref{r1.7}, for every pair $i, j$, we have $A_i\cap A_j\neq \emptyset$. Without loss of generality, we set $i=1$ and therefore, from the fact that, for all $2\leq j\leq n$, $A_1\cap A_j\neq \emptyset$ and that $\phi(A_1)=g_1\in G$, we conclude that  $\phi(A_j)=g_1\in G$. This means that every generator of $UVP_n$ is mapped to the same element $g_1\in G$, since $\bigcup_{i=1}^nA_i=\{\lambda_{i,j}\}_{1\leq i\neq j\leq n}$, from Remark \ref{r1.7}. Note that $g$ could possibly be the trivial element. From Theorem \ref{tss2} we have that $UVB_n$ is isomorphic to the semi-direct product $UVP_n\rtimes S_n$ and that for any generator $s\in S_n$ it holds that $s\lambda_{i,j}s^{-1}=\lambda_{s(i),s(j)}$. Under the homomorphism $\phi$ we obtain
	$\phi(s)\phi(\lambda_{i,j})(\phi(s))^{-1}=\phi(\lambda_{s(i),s(j)})$ and therefore $\phi(s)\phi(\lambda_{i,j})=\phi(\lambda_{s(i),s(j)})\phi(s)$. As a result, $\phi(s)g_1=g_1\phi(s)$. We conclude that the image of any generator $s$ of the symmetric group, $\phi(s)$, commutes with $g_1$. Thus, $\phi(UVB_n)\cong\mathbb{Z}_m\times \Image(\phi_{|S_n})$, for some $m\in \mathbb{Z}$.	 
\end{proof}

\begin{remark}
All possible finite image homomorphims of $WB_n$ and of the virtual braid groups have been determined by Scherich--Verberne in \cite{nancy}
\end{remark}

\section{The automorphism group of $UVP_n$}\label{s5}
In this section we shall determine the automorphism group of $UVP_n$. First, we introduce the notion of right-angled Artin groups. A right-angled Artin group, also known as graph group, is a group which admits a finite presentation in which the only relations are commuting relations among the generators.

Every right-angled Artin group defines a graph whose vertices are the generators of the group and for every two generators that commute there is an edge connecting these two vertices. The converse also holds. For every graph $\Gamma$, with $V$ its vertex set, there is a right-angled Artin group, graph group, associated to $\Gamma$, $R_{\Gamma}$, defined as follows:
$$R_{\Gamma}=\langle v\in V\ | \ uw=wu,\ \text{if}\ v, w\ \text{are joined by an edge in } \Gamma  \rangle.$$    
From this association we can see that the right-angled Artin group that corresponds to the complete graph on $n$ vertices is the free Abelian group $\mathbb{Z}$ and that the graph on $n$ vertices with no edges corresponds to the free group $F_n$ of rank $n$. For a general survey on the right-angled Artin groups we direct the reader to the article \cite{charney2007introduction} by Charney.

We can see that the group $UVP_n$ is a right-angled Artin group,  since it admits the following presentation, as stated in Theorem \ref{tss}:
$$UVP_n=\langle \lambda_{i,j}, \ 1\leq i\ne j \leq n \ | \ \lambda_{i,j}\lambda_{k,l}=\lambda_{k,l}\lambda_{i,j}, \ \text{for}\ (k,l)\neq(j,i),\ 1\leq i,j,k,l\leq n  \rangle.$$ 

From this presentation, the graph, $\Gamma$, which corresponds to the right-angled Artin group $UVP_n$, is a graph with $n(n-1)$ vertices, where the vertex set is $V=\{\lambda_{i,j}\}_{1\leq i\neq j\leq n}$ and there is an edge connecting every pair of vertices except for the pairs $\{\lambda_{i,j}, \lambda_{j,i}\}$, since these are only pairs of generators that do not commute.

We continue now with providing the theory around the automorphisms of graph groups.
Extending the work of Servatius \cite{servatius1989automorphisms}, a complete set of generators for the automorphism group of a graph group was found by Laurence \cite{laurence1995generating}. Before giving the main result we present some notions that will be needed.

Let $\Gamma$ be a graph with $V$ being its vertex set.

\begin{itemize}
	\item The link of a vertex $v\in V$, $lk(v)$, is the set of all vertices that are connected to $v$ with an edge. 
	\item The star of a vertex $v\in V$, $st(v)$, is the union $lk(v)\cup\{v\}$.
	\item For any $w\neq v$, $w,v\in V$, we say that $v$ dominates $w$, $w\leq v$, if $lk(w)\subseteq st(v)$.
\end{itemize}
The theorem that follows is due to Laurence \cite{laurence1995generating}, who proved the conjecture that had been stated, and in certain special cases proved by Servatius \cite{servatius1989automorphisms}.
\begin{theorem}[Laurence, \cite{laurence1995generating}]\label{gen}
	Let $\Gamma$ be a finite graph defining a graph group $R_{\Gamma}$. Then the following automorphisms generate the automorphism group of $R_{\Gamma}$, $Aut(R_{\Gamma})$:
	\begin{itemize}
		\item Inversions, $I_v:v\to v^{-1}$, which inverts a generator $v\in V$ and fix the rest.
		\item Dominated Transvections, $T_{v}:w\to wv$, for $w,v\in V$ such that $v$ dominates $w$, $w\leq v$, and fix the rest.
		\item Graph Automorphisms, $G$: Any bijection of the graph to itself that preserves the relation vertices-edges.
		\item Locally Inner Automorphisms: $L_{v, Y}:y\to vyv^{-1}$, for all $y\in Y$, where $Y$ is a connected component of \ $\Gamma\minus st(v)$ and \ $|Y|>1$.
	\end{itemize}
\end{theorem}

\begin{remark}\label{r3}
	The condition $|Y|>1$, in the Locally Inner Automorphisms, is placed in order to eliminate redundancy. This is so because in the case where $|Y|=1$, the automorphism $L_{v, Y}$ can be obtained by composition of dominated transvections (since the single vertex $y\in Y$ is dominated by v) and inversions. 
\end{remark}
Based on the graph that corresponds to $UVP_n$ we have the following remarks.

\begin{remark}\label{r4}
	The only domination relations that occur in the graph that corresponds to $UVP_n$ are $$\lambda_{i,j}\leq \lambda_{j,i}\ \text{and}\ \lambda_{j,i}\leq \lambda_{i,j}.$$
	This is the case because $lk(\lambda_{i,j})=V\setminus \{\lambda_{i,j}, \lambda_{j,i}\}\subseteq st(\lambda_{j,i})=V\setminus \{\lambda_{i,j}\}$, and similarly $lk(\lambda_{j,i})=V\setminus \{\lambda_{j,i}, \lambda_{i,j}\}\subseteq st(\lambda_{i,j})=V\setminus \{\lambda_{j,i}\}$. As a result, $\lambda_{i,j}$ and $\lambda_{j,i}$ dominate each other. It could not have been possible that $\lambda_{i,j}$ is dominated by another generator $\lambda_{k,l}\neq \lambda_{j,i}$, since  $lk(\lambda_{j,i})=V\setminus \{\lambda_{j,i}, \lambda_{i,j}\}\nsubseteq st(\lambda_{k,l})=V\setminus \{\lambda_{l,k}\}$. 
\end{remark}	 
\begin{remark}\label{r5}
	We observe that for every vertex $\lambda_{i,j}\in V$ the subgraph $\Gamma\setminus st(\lambda_{i,j})$ is just the vertex $\lambda_{j,i}$, since $st(\lambda_{i,j})=V\setminus \{\lambda_{j,i}\}$.
\end{remark} 

We shall now determine, for $n\geq 2$, the group $Aut(UVP_n)$.

\begin{proof}[Proof of Theorem \ref{th3}]
	Let $n\geq 2$, $1\leq i\neq j\leq n$  and $\Gamma$ be the graph associated to the right-angled Artin group $UVP_n$. We recall that the graph $\Gamma$, which corresponds to $UVP_n$, is a graph with $n(n-1)$ vertices, where the vertex set is $V=\{\lambda_{i,j}\}_{1\leq i\neq j\leq n}$ and there is an edge connecting every pair of vertices except for the pairs $\{\lambda_{i,j}, \lambda_{j,i}\}$. From Theorem \ref{gen} we see that the automorphism group of $UVP_n$ is generated by the following four families of automorphisms; the Inversions, the Dominated Transvections, the Graph Automorphisms and the Locally Inner Automorphisms. 
	
	From Remark \ref{r3} and  Remark \ref{r5} we conclude that in the case of $UVP_n$ we do not have any Locally Inner Automorphism. 
	
	From Remark \ref{r4} it follows that any Dominated Transvection in $UVP_n$ is generated by $T_{\lambda_{i,j}}$ and $T_{\lambda_{j,i}}$, and they are defined as follows: $$T_{\lambda_{i,j}}:\lambda_{j,i}\mapsto \lambda_{j,i}\lambda_{i,j},\ \text{while fixing the rest generators}$$ and $$T_{\lambda_{j,i}}:\lambda_{i,j}\mapsto \lambda_{i,j}\lambda_{j,i},\ \text{while fixing the rest generators}.$$ 
	
	It remains to determine the Inversions and the Graph Automorphisms of $UVP_n$. Clearly, the Inversions in $UVP_n$, are $I_{\lambda_{i,j}}$, where $$I_{\lambda_{i,j}}:\lambda_{i,j}\mapsto\lambda_{i,j}^{-1},\ \text{while fixing the rest generators}.$$
	We know that the size of the set the $\{\lambda_{i,j}\}_{1\leq i\neq j\leq n}$, is $n(n-1)$. Moreover the Inversion Automorphisms have order two. Therefore, for $I:=\{I_{\lambda_{i,j}} \}_{1\leq i\neq j\leq n}$, it follows that $$\langle I\rangle\cong\mathbb{Z}_2^{n(n-1)}.$$ 
	
	Finally, we will describe all possible Graph Automorphisms of the graph $\Gamma$. We recall that a Graph Automorphism is a symmetry of the graph, that is a bijection to itself while preserving the edge-vertex connectivity. In our case, we see that there are only two possible symmetries of our graph. One that exchanges the generators $\lambda_{i,j}$ and $\lambda_{j,i}$, while fixing the rest generators and another that exchanges the pair of generators $\{\lambda_{i,j}, \lambda_{j,i}\}$ with another pair of generators $\{\lambda_{k,l}, \lambda_{l,k}\}$, while fixing the rest generators. We denote these two Graph Automorphisms as follows:
	$$E_{i,j}:\lambda_{i,j}\leftrightarrow\lambda_{j,i},\ \text{while fixing the rest generators}$$
	and	
	$$P_{ij,kl}:\lambda_{i,j}\leftrightarrow \lambda_{k,l},\ P_{ij,kl}:\lambda_{j,i}\leftrightarrow\lambda_{l,k},\ \text{while fixing the rest generators},$$
	where $\lambda_{i,j}\leftrightarrow\lambda_{j,i}$ means that this map exchanges the generators $\lambda_{i,j}$ and $\lambda_{j,i}$.
	These two Graph Automorphisms are defined in such a way that it follows that $E_{i,j}=E_{j,i}$ and $P_{ij,kl}=P_{kl, ij}$.
	In $\Gamma$ all vertices are pairwise connected with an edge except for the pairs $\{\lambda_{i,j}, \lambda_{j,i}\}$, and the vertices $\lambda_{i,j}$ and $\lambda_{j,i}$ are connected with the exact same vertices. Therefore, if we want to preserve the edge-vertex connectivity we have to permute the vertex set in such a way that each vertex is connected with the same vertices before and after the permutation.
	Therefore, it is clear that only these two type of maps, $E_{i,j}$ and $P_{ij,kl}$, preserve the edge-vertex connectivity, and thus they are indeed the only symmetries of the graph. 
	
	For $E:=\{E_{i,j}\}$ and $P:=\{P_{ij,kl}\}$ we have that $|E|=n(n-1)/2$, since $E_{i,j}=E_{j,i}$ and that $|P|=n(n-1)/2$, since $P_{ij,kl}=P_{kl, ij}$. The elements $E_{i,j}\in E$ and $P_{ij,kl}\in P$ do not commute, since $(E_{i,j}P_{ij,kl})(\lambda_{i,j})=\lambda_{k,l}$ and $(P_{ij,kl}E_{i,j})(\lambda_{i,j})=\lambda_{l,k}$, where $\lambda_{k,l}\neq \lambda_{l,k}$. More precisely, we have that $P$ acts on $E$ by permuting the elements $E_{i,j}$ inside $E$. To be more precise, $P_{ij,kl}E_{i,j}P_{ij,kl}= E_{k,l}$, since $P_{ij,kl} E_{i,j}P_{ij,kl}(\lambda_{i,j})=\lambda_{i,j}$ and $P_{ij,kl}E_{i,j} P_{ij,kl}(\lambda_{k,l})=\lambda_{l,k}$. 
	It is easy to see that $\langle E \rangle\cong \mathbb{Z}_2^{n(n-1)/2}$ and that $\langle P \rangle\cong S_{n(n-1)/2}$. All together we obtain 
	$$\langle E,P\rangle\cong \mathbb{Z}_2^{n(n-1)/2}\rtimes S_{n(n-1)/2},$$
	where the symmetric group $S_{n(n-1)/2}$ acts on $\mathbb{Z}_2^{n(n-1)/2}$ by permuting the components of the product.
	
	Note that since $I_{\lambda_{i,j}}=E_{i,j}I_{\lambda_{j,i}}E_{i,j}$ it follows that 
	$$\langle I,E,P\rangle\cong \langle\mathbb{Z}_2^{n(n-1)/2}, \mathbb{Z}_2^{n(n-1)/2}\rtimes S_{n(n-1)/2}\rangle.$$
	Moreover, it holds that $E_{i,j}T_{\lambda_{j,i}}E_{i,j}=T_{\lambda_{i,j}}$, and thus we keep only the Dominated  Transvection $T_{\lambda_{j, i}}$ in the generating set of $Aut(UVP_n)$ and this completes the proof.  
\end{proof}

\begin{remark}
For $n=2$ it holds that $UVP_n\cong F_2=\langle x_1, x_2\rangle$ and from Theorem \ref{th3} it follows that $Aut(UVP_n)\cong \langle T_{\lambda_{2,1}}, \mathbb{Z}_2, \mathbb{Z}_2\rangle$, which is generated by $T_{\lambda_{2,1}}, I_{\lambda_{2,1}}, E_{1,2}$. By Nielsen in \cite{nielsen1924isomorphismengruppe} it follows that $Aut(F_2)$ is generated by the automorphisms
$\alpha_1:x_1\leftrightarrow x_2,$
$\alpha_2:x_1\mapsto x_1^{-1},\ x_2\mapsto x_2,$
$\alpha_3:x_1\mapsto x_1x_2,\ x_2\mapsto x_2$. It is clear that the automorphisms $T_{\lambda_{2,1}}, I_{\lambda_{2,1}}, E_{1,2}$ correspond to the automorphisms $\alpha_1, \alpha_2, \alpha_3$ respectively. Thus, for $n=2$ we conclude that the automorphisms given in Theorem \ref{th3} generate $Aut(F_2)$.
\end{remark}

Having a concrete set of generators of the automorphism group of $UVP_n$ we can analyse this result a bit further. We recall that, from Remark \ref{free}, the group $UVP_n$ is isomorphic to the direct product of $n(n-1)/2$ copies of the free group of rank 2:
$$UVP_n=\langle \lambda_{1,2}, \lambda_{2,1} \rangle\times\dots\times\langle \lambda_{i,j},\lambda_{j,i} \rangle\times\dots\times\langle \lambda_{n-1,n},\lambda_{n,n-1} \rangle, \ \text{for}\ 1\leq i\neq j\leq n,$$
$$UVP_n\cong\underbrace{F_2\times\dots\times F_2\times\dots\times F_2}_\text{$n(n-1)/2$-times}, \ \text{for}\ n\geq 2.$$\\
Based on this isomorphism, we can make some observations about the generators of the automorphism group of $UVP_n$, given in Theorem \ref{th3}.
For $1\leq i\neq j\leq n$ the group $Aut(UVP_n)$ is generated by the following four automorphisms:\\\\
$T_{\lambda_{j,i}}:\lambda_{i,j}\mapsto \lambda_{i,j}\lambda_{j,i},\ \text{while fixing the rest generators},$\\
$I_{\lambda_{i,j}}:\lambda_{i,j}\mapsto\lambda_{i,j}^{-1},\ \text{while fixing the rest generators},$\\
$E_{i,j}:\lambda_{i,j}\leftrightarrow\lambda_{j,i},\ \text{while fixing the rest generators},$\\
$P_{ij,kl}:\lambda_{i,j}\leftrightarrow \lambda_{k,l},\ P_{ij,kl}:\lambda_{j,i}\leftrightarrow\lambda_{l,k},\ \text{while fixing the rest generators}.$\\\\
With the exception of the automorphisms $P_{ij,kl}$, it follows that all these automorphisms do not permute the $F_2$-factors of $UVP_n$, but they rather take an element from a factor and send it to the same factor. In other words, each $F_2$-factors of $UVP_n$ stay invariant under the automorphisms $T_{\lambda_{j,i}}$, $I_{\lambda_{i,j}}$ and $E_{i,j}$. Moreover, the image of a generator $\lambda_{i,j}$, under these automorphisms, belongs to the set generated by the elements $\{\lambda_{i,j}\lambda_{j,i}, \lambda_{i,j}^{-1}, \lambda_{j,i} \}$. These are the three automorphisms that generate the group $Aut(F_2)$, as we will see shortly. On the contrary, the automorphism $P_{ij,kl}$ permutes the $n(n-1)/2$ $F_2$-factors of $UVP_n$.  

It is well known, and proved by Nielsen in \cite{nielsen1924isomorphismengruppe}, that $Aut(F_2)$, for $F_2=\langle x_1, x_2\rangle$, is generated by the following three automorphisms:\\
$\alpha_1:x_1\leftrightarrow x_2,$\\
$\alpha_2:x_1\mapsto x_1^{-1},\ x_2\mapsto x_2,$\\
$\alpha_3:x_1\mapsto x_1x_2,\ x_2\mapsto x_2.$

It turns out that $Aut(UVP_n)$ is the semi-direct product of $Aut(F_2)^{n(n-1)/2}$ and the symmetric group $S_{n(n-1)/2}$, which acts on $Aut(F_2)^{n(n-1)/2}$ by permuting the  $n(n-1)/2$ factors of $Aut(F_2)$. And so we obtain the following corollary.

\begin{corollary}\label{coinc}
	For $n\geq 2$ it holds that $$Aut(UVP_n)\cong Aut(F_2)^{{n(n-1)}/2}\rtimes S_{n(n-1)/2},$$ 
	where $S_{n(n-1)/2}$ acts on $Aut(F_2)^{n(n-1)/2}$ by permuting the  $n(n-1)/2$ $F_2$-factors.
\end{corollary}

This result agrees with a particular case of a more general result proved by Zhang--Ventura--Wu in \cite{zhang2015fixed}, where they obtained the same result using different techniques. We state their result below and we see that for $m=1$ and $n_1=n(n-1)/2$ their result coincides with Corollary \ref{coinc}.

\begin{proposition}[Zhang--Ventura--Wu, \cite{zhang2015fixed}]
	Let $G=G_1^{n_1}\times\dots\times G_m^{n_m}$ be a product group, where $m\geq1$, $n_i\geq1$, $G_i\ncong G_j$, for $i\neq j$, and each $G_i$ is a free group or a surface group. If $G$ is a hyperbolic type, that is $Z(G_i)=1$ for all $i=1, \dots, m$, then for every $\phi\in Aut(G)$ there exists automorphisms $\phi_{i,j}\in Aut(G_i)$ and permutations $\sigma_i\in S_{n_i}$, such that 
	$$\phi=\sigma_1\circ\dots\circ\sigma_m\circ\big(\prod_{i=1}^{m}\prod_{j=1}^{n_i}\phi_{i,j}\big)=\prod_{i=1}^{m}\big(\sigma_i\circ\prod_{j=1}^{n_i}\phi_{i,j}\big).$$
\end{proposition}

\section{About the automorphism group of $UVB_n$}\label{s6}
We shall now present partial results about the automorphism group of $UVB_n$. In \cite{rose1975automorphism}, Rose gave a description of the automorphism group of groups which possess a proper characteristic subgroup that have trivial centraliser.
\begin{proposition}[Rose, \cite{rose1975automorphism}]\label{rose}
	Let $G$ be a group with a characteristic subgroup $H$ such that $C_G(H)=e$. Then $G$ is naturally embedded in $Aut(H)$ by means of conjugation of $H$ by the elements of $G$. Moreover, there is a natural isomorphism between $Aut(G)$ and the normaliser of $G$ in $Aut(H)$. That is $Aut(G)\cong N_{Aut(H)}(G)$. 
\end{proposition}
We proved that $UVP_n$ is a characteristic subgroup of $UVB_n$, Proposition \ref{char}, and that $C_{UVB_n}(UVP_n)=e$, Proposition \ref{centralizer}. Therefore, applying Proposition \ref{rose} for $G=UVB_n$ and $H=UVP_n$, we obtain the following corollary.

\begin{corollary}\label{colrose}
	For $n\geq 5$ it holds that 
	$$Aut(UVB_n)\cong N_{Aut(UVP_n)}(UVB_n).$$
\end{corollary}

Moreover, we will present a subgroup of $Aut(UVB_n)$. Consider the maps $\beta_n$ and $\gamma_n$ defined as follows:
\begin{equation}\label{beta}
\beta_n: 
\begin{cases}
\sigma_i\mapsto \sigma_i^{-1},\ \text{for}\ 1\leq i\leq n-1,\\
\rho_i\mapsto \rho_i,\ \text{for}\ 1\leq i\leq n-1,
\end{cases}
\end{equation}
and
\begin{equation}\label{gamma}
\gamma_n: 
\begin{cases}
\sigma_i\mapsto \rho_i\sigma_i\rho_i,\ \text{for}\ 1\leq i\leq n-1,\\
\rho_i\mapsto \rho_i,\ \text{for}\ 1\leq i\leq n-1.
\end{cases}
\end{equation}
One can verify that these two maps are actually automorphisms of $UVB_n$. In particular, note that $\beta_n$ and $\gamma_n$ are of order two and moreover $\beta_n\circ\gamma_n=\gamma_n\circ\beta_n$. Thus, $\langle \beta_n, \gamma_n\rangle$ generates
a subgroup of $Aut(UVB_n)$ isomorphic to $\mathbb{Z}_2\times\mathbb{Z}_2$. Furthermore, as we will see in the following proposition, the automorphisms $\beta_n$, $\gamma_n$ and $\beta_n\circ\gamma_n$ are actually elements of the outer automorphisms group of $UVB_n$, $Out(UVB_n)$, where $Out(UVB_n)=Aut(UVB_n)/Inn(UVB_n)$.

\begin{proof}[Proof of Proposition \ref{prop1}]
	Let $n\geq 3$. We will show that the automorphism $\beta_n$, $\gamma_n$ and $\beta_n\circ\gamma_n$ are not inner automorphisms. Note that any inner automorphism of a group $G$ acts trivially on the Abelianisation of $G$. That is if $h: G\to G$ is an inner automorphism then the induced automorphism $\bar{h}: G^{ab}\to G^{ab}$ is the trivial one. From Remark \ref{useouter}, we have that the Abelianisation of $UVB_n$ is isomorphic to $\mathbb{Z}\times\mathbb{Z}_2$, where $\mathbb{Z}$ is generated by the image of $\sigma_1$, $[\sigma_1]$ and $\mathbb{Z}_2$ is generated by the image of $\rho_1$, $[\rho_1]$. The automorphisms $\beta_n$ and $\beta_n\circ\gamma_n$ act non-trivially on the Abelianisation of $UVB_n$, since $\beta_n(\sigma_1)=\sigma_1^{-1}$ and $(\beta_n\circ\gamma_n)(\sigma_1)=\rho_1\sigma_1^{-1}\rho_1$, for $1\leq i\leq n-1$, and thus they are not inner automorphisms. 
	
	By contradiction we will prove that $\gamma_n$ is not an inner automorphism. First, notice that from the set of Relations (\ref{lamda}) we have that $\gamma_n(\lambda_{k,l})=\lambda_{l,k}$, for every $\lambda_{k,l}\in UVP_n$, where $k\neq l$ and $k,l=1,\dots,n$. Suppose that $\gamma_n$ is an inner automorphism. It follows that there exists some non-trivial element $g\in UVB_n$ such that $\gamma_n(u)=i_g(u)=gug^{-1}$, for every $u\in UVB_n$. In particular, it has to hold that $\gamma_n(\lambda_{k,l})=g\lambda_{k,l}g^{-1}=\lambda_{l,k}$, for every $k,l=1,\dots,n$, where $k\neq l$.
	Using the same arguments as we did in the proof of Proposition \ref{centralizer},  we will show that the relation $\lambda_{l,k}=g\lambda_{k,l}g^{-1}$, for every $k,l=1,\dots,n$, $k\neq l$ and for a fixed non-trivial element $g\in UVB_n$, can not hold.  
	
	We know that $UVB_n=UVP_n\rtimes S_n$ and thus,  $g=\Lambda S$, for $\Lambda, S$ fixed words in $UVP_n\ \text{and}\ S_n$ respectively. Therefore, it has to hold that 
	\begin{equation*}
	\lambda_{l,k}\cdot \Lambda S=\Lambda S\cdot \lambda_{k,l},\ \text{for every}\ k\neq l\in \{1,\dots,n\}.
	\end{equation*}
	From Theorem \ref{tss2}, and in particular based on the action of the symmetric group on every generator $\lambda_{k,l}\in UVP_n$, we obtain
	
	$$\lambda_{l,k}\cdot \Lambda S=\Lambda\cdot\lambda_{S(k),S(l)}S,\ \text{for every}\ k\neq l\in \{1,\dots,n\}.$$
	Thus,
	\begin{equation}\label{Equal2}
	\lambda_{l,k}=\Lambda\cdot\lambda_{S(k),S(l)}\cdot \Lambda^{-1} \in UVP_n,\ \text{for every}\ k\neq l\in \{1,\dots,n\}.
	\end{equation}
	
	Suppose that either $S(k)\neq l$ or $S(l)\neq k$. Then, from Relation \eqref{Equal2}, we see that under the Abelianisation map these distinct generators, $\lambda_{l,k}$ and $\lambda_{S(k),S(l)}$, of $UVP_n$, would coincide. But this leads to a contradiction since we know from Corollary \ref{abuvpn} that the Abelianisation of $UVP_n$ is isomorphic to the free Abelian group of rank $n(n-1)$ generated by the images of the elements $\lambda_{i, j}$, for $1\leq i\neq j\leq n$. Therefore, it has to hold that $S(k)=l$ and $S(l)=k$ for every $k,l\in \{1,\dots,n\}$, $k\neq l$, where $S$ is a fixed element in $S_n$. But once again, this leads to a contradiction because a permutation on $n$ elements cannot permute all possible couples $(l,k)$, for $1\leq l \neq k \leq n$.
	
	Now, suppose that $g=S$, meaning that $\Lambda$ is a trivial word in $UVP_n$. Then, it has to hold that
	\begin{equation}\label{Equal3}
	\lambda_{l,k}\cdot S=S\cdot \lambda_{k,l},\ \text{for every}\ k\neq l\in \{1,\dots,n\}.
	\end{equation} 
	But Relation \eqref{Equal3} implies that the word $S$ in the symmetric group fixes all the elements of the set $\{1,\dots,n\}$. This is possible only when $S$ is the trivial element, which leads once more to a contradiction, since $g$ is a non-trivial element.
	Since the centre of $UVP_n$ is trivial, Remark \ref{center}, we do not need to check the case where $g=\Lambda$.
	
	Thus, we conclude that $\gamma_n$ is not an inner automorphism and this completes the proof.
\end{proof}

We speculate that this result about the outer automorhism group of $UVB_n$, $Out(UVB_n)$, could be of help in determining the group $Out(WB_n)$, which is still an open problem. Nevertheless, we conjecture that the outer automorphism group of $WB_n$, which consists of non-inner automorphisms, is not trivial, and more precisely, we conjecture that, for $n\geq 3$, the group $Out(WB_n)=Aut(WB_n)/Inn(WB_n)$ is generated by the automorphism $\alpha_n$, which is of order two and is defined as follows:
\begin{equation}\label{auto}
\alpha_n: 
\begin{cases}
\sigma_i\mapsto \rho_i\sigma_i^{-1}\rho_i,\ \text{for}\ 1\leq i\leq n-1,\\
\rho_i\mapsto \rho_i,\ \text{for}\ 1\leq i\leq n-1.
\end{cases}
\end{equation}
Therefore, we conjecture that, for $n\geq 3$, $Out(WB_n)\cong\mathbb{Z}_2$.

Note that the automorphism $\alpha_n$, defined in \eqref{auto}, can be seen as the composition of the maps $\beta_n$ and $\gamma_n$, defined in \eqref{beta} and \eqref{gamma}, which are elements of $Out(UVB_n)$.
These two maps, although their composition is an automorphism of $WB_n$, they are not automorphisms of the group $WB_n$. Actually, they are not even homomorphisms of $WB_n$, since relation \ref{rel8}, $\rho_i\sigma_{i+1}\sigma_i=\sigma_{i+1}\sigma_i\rho_{i+1}$, for $i=1, \dots, n-2$, in $WB_n$, is not preserved neither under the map $\beta_n$ nor under the map $\gamma_n$. 

Finally, we complete this section by proving that the groups $UVB_n$ and $UVP_n$ are residually finite and Hopfian, but not co-Hopfian.

We recall that a group $G$ is called Hopfian if every surjective homomorphism $G\to G$ is also an injective homomorphism and it is called co-Hopfian if every injective homomorphism $G\to G$ is also a surjective homomorphism.

\begin{corollary}\label{hop}
	Let $n\geq 2$. The groups $UVB_n$ and $UVP_n $ are residually finite and Hopfian, but not co-Hopfian.
\end{corollary}

\begin{proof}
	By Theorem \ref{tss2}, we have that $UVB_n\cong UVP_n\rtimes S_n$. We already know that the group $UVP_n$ is a right-angled Artin group and it is known that every right-angled Artin group is linear; see \cite{hsu1999linear}. Moreover, it is also known that a finitely generated linear group is residually finite. It now follows that $UVP_n$ is
	a residually finite group. Moreover, $UVB_n$ is a residually finite group, since it is an extension of $UVP_n$ by $S_n$, which is a finite group.
	Finally, the groups $UVB_n$ and $UVP_n$ are Hopfian, since they are finitely generated, residually finite groups. 
	
	In order to show that the group $UVP_n$ is not co-Hopfian, we just provide the following homomorphism:
	$$h:UVP_n\to UVP_n \ \text{defined by}\ h:\lambda_{i,j}\mapsto \lambda_{i,j}\lambda_{j,i},\ \text{for every}\ 1\leq i\neq j\leq n,$$
	which is injective but not surjective, since the elements $\lambda_{i,j}$, for $1\leq i\neq j\leq n$ do not have a preimage under $h$. Now, we extend the map $h$ to a homomorphism of $UVB_n\cong UVP_n\rtimes S_n$ as follows: $$\bar{h}:UVB_n\to UVB_n\ \text{defined by}\ \bar{h}:\lambda_{i,j}\mapsto \lambda_{i,j}\lambda_{j,i}\ \text{and} \ \bar{h}:s_k\mapsto s_k,$$
	for every $1\leq i\neq j\leq n\ \text{and}\ 1\leq k\leq n-1,\ \text{where}\ s_k\in S_n.$
	Similarly, this homomorphism is injective but not surjective, since the elements $\lambda_{i,j}$, for $1\leq i\neq j\leq n$ do not have a preimage under $\bar{h}$, and therefore the group $UVB_n$ is not co-Hopfian. 
\end{proof}

\section*{Acknowledgement}	
Thanks are due to Paolo Bellingeri for helpful conversations and comments on this work.

\bibliographystyle{plain}
\bibliography{bibliography.bib}
\end{document}